%% file: ms-semi-springer.tex
\documentclass[graybox]{svmult}

\usepackage{type1cm}        % activate if the above 3 fonts are
\usepackage{makeidx}         % allows index generation
\usepackage{graphicx}        % standard LaTeX graphics tool
\usepackage{multicol}        % used for the two-column index
\usepackage[bottom]{footmisc}% places footnotes at page bottom

\usepackage{newtxtext}       % 
\usepackage[varvw]{newtxmath}       % selects Times Roman as basic font

\makeindex             % used for the subject index
\usepackage{dsfont,stmaryrd}
\usepackage{mathtools}
\usepackage[mathcal]{euscript}
\usepackage{bm}
\usepackage{import}
\usepackage{xifthen}
\usepackage{enumitem}

\newcommand{\NN}{{\mathbb N}}
\newcommand{\ZZ}{{\mathbb Z}}

\newcommand{\RR}{{\mathbb R}}

\newcommand{\bfa}{\mathbf{a}}

\newcommand{\bfr}{\mathbf{r}}

\newcommand{\nn}{\mathbf{n}}
\newcommand{\ee}{\mathrm{e}}

\newcommand{\per}{\mathrm{per}}%{\sharp}%
\newcommand{\sym}{\mathrm{sym}}

\newcommand{\hmeas}{\mathcal{H}^{d-1}}

\newcommand{\abs}[1]{\left\lvert #1\right\rvert}

\newcommand{\norm}[1]{\left\lVert #1\right\rVert}
\newcommand{\tscale}{\xrightharpoonup[]{~2~}}%{\overset{2}{\rightharpoonup}}
\newcommand{\cv}[1][]{%
\ifthenelse{\isempty{#1}}{\xrightarrow[\hphantom{~2~}]{}}{\xrightarrow[\hphantom{~2~}]{#1}}%
}
\newcommand{\wcv}[1][]{%
\ifthenelse{\isempty{#1}}{\xrightharpoonup[\hphantom{~2~}]{}}{\xrightharpoonup[\hphantom{~2~}]{#1}}%
}
\newcommand{\coex}[2][\hphantom{longexplanation}]{%
  \overset{\text{#1}}{\underset{\hphantom{longexplanation}}{#2}}
}
\newcommand{\comp}[2]{%
  \ifthenelse{\isempty{#1}}
  {\overset{\hphantom{longexplanation}}{#2}}
  {\overset{\text{#1}}{\underset{\hphantom{longexplanation}}{#2}}}
}

\newcommand{\calD}{\mathcal{D}}

\newcommand{\fraM}{\mathfrak{M}}

\DeclareMathOperator{\dd}{d\!}
\DeclareMathOperator{\e}{{e}}
\DeclareMathOperator{\Div}{{div}}

\def\XXint#1#2#3{{\setbox0=\hbox{$#1{#2#3}{\int}$ }
\vcenter{\hbox{$#2#3$ }}\kern-.6\wd0}}

\def\coloneqq{\doteq}

\DeclarePairedDelimiterX\Set[1]\{\}{%
  #1%
}

\begin{document}

\title*{Homogenization of a semilinear elliptic problem}
\author{Thuyen Dang\orcidID{0000-0003-4370-2400}
  and\\ Yuliya Gorb\orcidID{0000-0002-9968-4494}
  and\\ Silvia Jim\'{e}nez Bola\~{n}os\orcidID{0000-0002-3086-5748}}
\institute{Thuyen Dang \at University of Chicago, 5747
    S. Ellis Avenue, Chicago, Illinois 60637, \email{thuyend@uchicago.edu}
\and Yuliya Gorb \at National Science Foundation, 2415 Eisenhower Avenue, Alexandria, Virginia 22314
\email{ygorb@nsf.gov}
\and Silvia Jim\'{e}nez Bola\~{n}os \at Colgate University, 13 Oak Drive, Hamilton, New York 13346 \email{sjimenez@colgate.edu}}
\maketitle

\abstract*{We consider the homogenization of a semilinear elliptic equation where the coefficients of the second-order differential operator may be discontinuous. We establish the existence and uniqueness of the fine-scale solution, followed by an a priori estimate. The homogenized equation is derived using two-scale convergence, and a corrector result is also provided.}

\abstract{We consider the homogenization of a semilinear elliptic
  equation where the coefficients of the second-order differential
  operator may be discontinuous. We establish the existence and
  uniqueness of the fine-scale solution, followed by an a priori
  estimate. The homogenized equation is derived using two-scale
  convergence, and a corrector result is also provided.}

\section{Introduction}
\label{sec:intro}
%This paper is concerned with a homogenization problem of a semilinear elliptic equations, where the coefficients of the second order differential operator may be discontinuous. The original problem without the fine-scale was consider in \cite{brezisSemilinearSecondorderElliptic1973} and reference therein. On the application side, the motivation for such kind of problem was briefly surveyed in \cite{donatoHomogenizationQuasilinearEl iptic2018}.

The study of semilinear and quasilinear second-order elliptic partial differential equations is fundamental to modeling a wide range of phenomena in physics and engineering, including heat transfer, diffusion processes, and mechanics. Specific applications and examples are discussed in
\cite{cazenaveSemilinearSchrodingerEquations2003,ablowitzNonlinearDispersiveWaves2011,badialeSemilinearEllipticEquations2011,donatoHomogenizationQuasilinearElliptic2018}
and the references therein; among them, materials such as glass and wood (with conductivity increasing nonlinearly with temperature), ceramics (with decreasing conductivity), and aluminum or semiconductors (exhibiting non-monotone behavior).

% The paper \cite{brezisSemilinearSecondorderElliptic1973} investigates semilinear second-order elliptic equations with $L^1$-data, establishing the existence and uniqueness of solutions under various conditions, including both monotone and non-monotone nonlinearities, as well as nonlinear boundary conditions. When the source terms are in $L^1(\Omega)$, standard variational techniques, relying on energy estimates in Hilbert spaces, are generally not applicable, necessitating specialized methods beyond traditional variational frameworks. The authors introduce an alternative approach based on multiplication by monotone functions and abstract operator theory, drawing connections to related works, and extending the analysis to Orlicz spaces.

In many physical scenarios, the materials or media through which these processes occur exhibit complex internal structures that can be heterogeneous at a microscopic level. For example, composite materials, porous media, and biological tissues possess intricate geometries and varying material properties across small scales. When these microstructures are periodic, understanding the effective macroscopic behavior of the system becomes important as the scale of the heterogeneities, denoted by $0 < \varepsilon \ll 1$, becomes much smaller than the overall domain. This is the realm of homogenization theory, which provides a framework for deriving macroscopic equations that govern the average behavior of a microscopically heterogeneous material.

The paper \cite{donatoHomogenizationQuasilinearElliptic2018} examines the homogenization of quasilinear elliptic problems with nonlinear Robin boundary conditions and $L^1$-data in periodically perforated domains. Using the periodic unfolding method and the concept of renormalized solutions to handle the lack of regularity of the $L^1$-data, the authors sought to describe the macroscopic behavior of such heterogeneous materials by deriving both an unfolded renormalized problem and a homogenized renormalized problem.

Unlike the previous studies mentioned above, the main challenge in the semilinear problem considered in this paper arises from the nonlinearities, as we assume a regular source term. A key motivation for studying such semilinear elliptic equations is that they represent steady-state solutions of (i) the nonlinear Schrödinger equation, which models wave propagation in composite fiber optics, and (ii) reaction–diffusion systems, whose one prototype is the KPP-Fisher equation modeling several natural phenomena, \cite{cazenaveSemilinearSchrodingerEquations2003,ablowitzNonlinearDispersiveWaves2011,grindrodTheoryApplicationsReactiondiffusion1996}. Inspired by the works of \cite{brezisSemilinearSecondorderElliptic1973,balFluctuationsHomogenizationSemilinear2016,donatoHomogenizationQuasilinearElliptic2018}, this paper aims to advance the understanding of semilinear elliptic equations with discontinous coefficients in the second-order differential operator, with regular data, and nonlinear terms by studying the homogenization of a class of second-order semilinear equations. The focus is on the challenges arising from monotone nonlinearities and periodic heterogeneities.

The main results of the paper include: (i) the existence and uniqueness of solutions to the original fine-scale problem, along with an a priori error estimate under a specific growth condition on the monotone nonlinearity, enabling the use of the compactness method in the study of homogenization; and (ii) the derivation of the effective (homogenized) response of the system, as well as a first-order corrector, under additional regularity assumptions on the source term.

The paper is organized as follows. Section \ref{sec:formulation} presents the formulation of the problem. The main results and their proofs are provided in Section \ref{sec:main-results}. Conclusions are given in Section~\ref{sec:Conclusions} and auxiliary facts and key definitions are included in %\cref{sec:some-results-two}.

\section{Formulation}
\label{sec:formulation}
\subsection{Notational conventions}
% Throughout this paper, scalar-valued functions, such as $h$ in \eqref{eq:p423}, are written in usual fonts, while vector-valued or tensor-valued functions, such as the displacement $\uu$ and the elastic tensor $\bfR$ in \eqref{eq:5}, are written in bold.  Sequences are indexed by superscripts ($\phi^i$), while elements of vectors or tensors are indexed by numeric subscripts ($x_i$).
We denote by $\mathds{1}_{A}$ the characteristic function of the set $A$. Finally, the Einstein
summation convention is used whenever applicable; $\delta_{ij}$ is the Kronecker delta, and $\epsilon_{ijk}$ is the permutation symbol.

\subsection{Set up of the problem}
\label{ss:setup}
Consider $\Omega \subset \RR^d$, for $d \ge 2$, an open bounded set with smooth boundary, and let
$Y\coloneqq [0,1]^d$ be the unit cell in $\RR^d$. The unit cell $Y$ is decomposed into:
$$Y=Y_s\cup Y_f \cup \Gamma,$$
where $Y_s \subset\subset (0,1)^d$, representing the domain occupied by the inclusion, and $Y_f$,
representing the matrix material, are open sets in
$\mathbb{R}^d$, and $\Gamma$ is the closed $C^{1,1}$ interface that
separates them. Let $i = (i_1, \ldots, i_d) \in \ZZ^d$ be a vector of indices and $\{\ee^1, \ldots, \ee^d\}$ be the canonical basis of $\RR^d$. 
For a fixed small $\varepsilon > 0,$ we define the dilated sets: 
\begin{align*}
    Y^\varepsilon_i 
    \coloneqq \varepsilon (Y + i),~~
    Y^\varepsilon_{i,s}
    \coloneqq \varepsilon (Y_s + i),~~
    Y^\varepsilon_{i,f}
    \coloneqq \varepsilon (Y_f + i),~~
    \Gamma^\varepsilon_i 
    \coloneqq \partial Y^\varepsilon_{i,s}.
\end{align*}
Typically, in homogenization theory, the positive number $\varepsilon
\ll 1$ is referred to as the {\it size of the microstructure}. The
effective or homogenized response of the given composite
corresponds to the case $\varepsilon=0$, whose derivation and justification is the main focus of this paper. 

We denote by $\nn_i,~\nn_{\Gamma}$ and $\nn_{\partial \Omega}$ the unit normal vectors to $\Gamma^{\varepsilon}_{i}$ pointing outward $Y^\varepsilon_{i,s}$, on $\Gamma$ pointing outward $Y_{s}$ and on $\partial \Omega$ pointing outward, respectively; and also, we denote by $\dd \hmeas$ the $(d-1)$-dimensional Hausdorff measure.
In addition, we define the sets:
\begin{align}
  \label{eq:116}
    I^{\varepsilon} 
    \coloneqq \{ 
    i \in \ZZ^d \colon Y^\varepsilon_i \subset \Omega
    \},~~
    \Omega_s^{\varepsilon} 
    \coloneqq \bigcup_{i\in I^\varepsilon}
Y_{i,s}^{\varepsilon},~~
    \Omega_f^{\varepsilon} 
    \coloneqq \Omega \setminus \Omega_s^{\varepsilon},~~
    \Gamma^\varepsilon 
    \coloneqq \bigcup_{i \in I^\varepsilon} \Gamma^\varepsilon_i.
\end{align}
see Figure \ref{fig:1}.

\begin{figure}[ht]
\centering
\def\svgwidth{0.5\columnwidth}
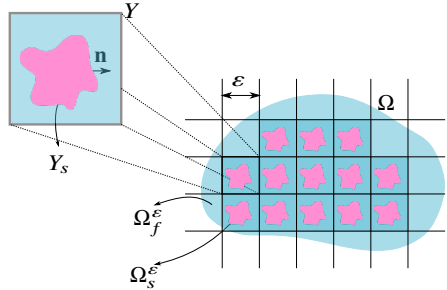
\caption{Reference cell $Y$ and domain $\Omega$.}
\label{fig:1}
\end{figure}

Given a matrix $\bfr$ of size $d\times d$, we introduce the following definitions:
\begin{enumerate}[label=(A{\arabic*}),ref=\textnormal{(A{\arabic*})}]
\item \label{cond:a-periodic} $Y-$periodicity: for all $z
  \in \RR^d$, for all $m \in \ZZ$, and for all $k \in \{ 1,\ldots, d \}$ we have 
\begin{equation*}
%\label{eq:3}
\bfr (z + m \ee^k) = \bfr (z).
\end{equation*}
  \item\label{cond:a-bound} Boundedness: there exists $\Lambda > 0$ such that
\begin{equation*}
%\label{eq:2}
\norm{\bfr}_{L^{\infty}} \le \Lambda.
\end{equation*}
  \item\label{cond:a-elliptic} Ellipticity: there exists
    $\lambda > 0$ such that for all $\xi \in \RR^d$, for all $x \in
    \RR^d$, we have
\begin{equation*}
%\label{eq:1}
\bfr (x) \xi \cdot \xi \ge \lambda \abs{\xi}^2.
\end{equation*}
% \item \label{cond:a-holder-cont} Piecewise H\"{o}lder continuity: there exists $\alpha \in (0,1)$ and a
%   partition $\left\{ D_i \right\}_{i=1}^n$ of $Y$ such that $\bfr$ is
%   $C^{\alpha}-$H\"{o}lder continuous on each $D_i$, $1 \le i \le n$.\\
\end{enumerate}

Denote by $\fraM (\lambda,\Lambda)$ the set of matrices that satisfy
\ref{cond:a-bound}-\ref{cond:a-elliptic}, and by
$\fraM_{\sym,\per}(\lambda,\Lambda)$ the subset of symmetric matrices in $\fraM(\lambda,\Lambda)$
that also satisfy \ref{cond:a-periodic}.

Let $\bfa \in \fraM_{\sym,\per}(\lambda,\Lambda)$. For each $\varepsilon > 0$, define
$\bfa^{\varepsilon}(x) \coloneqq \bfa \left( \frac{x}{\varepsilon}
\right)$.

Let $2_{*} \coloneqq
\frac{2d}{d+2}$ and $2^{*} \coloneqq \frac{2d}{d-2}$ if $d \ge 3$ and let $2^{*}$ be an arbitrary number $\ge 2$ if $d = 2$. 
Let $g\colon \RR \to \RR$
be continuous,  non-decreasing, and satisfying the following growth condition
\begin{align}
\label{eq:16}
\abs{g(u)} \le C_{g} \left( \abs{u}^{q-1} + h \right),
\end{align}
for $u \in \RR$, $ C_{g} > 0$, $1 \le q \le 2^{*}$,  and $h \in
L^{2_{*}}(\Omega)$.

The fact that $g$ is increasing implies that
\begin{align}
\label{eq:5}
\left( g(u) - g (v) \right)(u-v) \ge 0,
\end{align}
for all $u,v \in \RR$.
 % Bounded from below together with
%   \eqref{eq:5} seem very restrictive. We don't actually need bounded
%   from below, but only $g(u_{mn}^{\varepsilon})v_k \ge V$ for some $V \in L^1(\RR)$, for
%   the argument using Fatou Lemma.

%   Let $v = 0$ in  \eqref{eq:5}, we obtain $\abs{g(u)}\abs{u} = g(u)u
%   \ge c_g \abs{u}^q$, therefore, 
% \begin{align}
% \label{eq:27}
% \abs{g(u)} \ge c_g \abs{u}^{q-1}.
  % \end{align}

Let $f \in L^{2_*}(\Omega)$. Let $\varepsilon > 0$. We consider the following boundary value problem:
  \begin{align}
    \label{eq:4a}
    \begin{cases}
  -\Div \left( \bfa^{\varepsilon}\nabla u^{\varepsilon} \right) +
    g ({u^{\varepsilon}}) 
    &= f \qquad \text{ in }\Omega,\\
    u^{\varepsilon}
    &= 0 \qquad \text{ on }\partial \Omega,
    \end{cases}
  \end{align}
  we are interested in understanding what happens as $\varepsilon\to0$.

\subsection{Two-scale convergence}
\label{sec:some-results-two}
The proof of the main result in this paper will be based on the theory of two-scale convergence, conceived by G. Nguetseng \cite{nguetsengGeneralConvergenceResult1989} and was
further developed by G. Allaire \cite{allaireHomogenizationTwoscaleConvergence1992}.  Here, we collect important notions
and results relevant to this paper, whose proofs can be found in \cite{allaireHomogenizationTwoscaleConvergence1992,nguetsengGeneralConvergenceResult1989}. The following spaces are used later in this paper.
\begin{itemize}[wide]
    \item $C_{\per}(Y)$ -- the subspace of $C(\RR^d)$ of $Y$-periodic functions;
    \item $C^{\infty}_{\per}(Y)$ -- the subspace of $C^{\infty}(\RR^d)$ of $Y$-periodic functions;
    \item $H^1_{\per}(Y)$ -- the closure of $C^{\infty}_{\per}(Y)$ in the $H^1$-norm;
    \item
    
    $\mathcal{D}(\Omega, X)$ -- where $X$ is a Banach space -- the space  infinitely differentiable functions from $\Omega$ to $X$, whose  support is a compact set of $\mathbb{R}^d$ contained in $\Omega$.

    \item $L^p(\Omega, X)$ -- where $X$ is a Banach space and $1 \le p \le \infty$ -- the space of measurable functions $w \colon x \in \Omega \mapsto w(x) \in X$ such that
    $
    \norm{w}_{L^p(\Omega, X)}
    \coloneqq \left(\int_{\Omega} \norm{w(x)}^p_{X} \dd x\right)^\frac{1}{p} < \infty.
    $
     
    \item $L^p_{\per}\left(Y, C(\bar{\Omega})\right)$ -- the space of measurable functions $w \colon y \in Y \mapsto w(\cdot,y) \in C(\bar{\Omega})$, such that %$\norm{w(x,\cdot)}_{C(\bar{\Omega})} \in L^2_{\per}(Y)$.
    $w$ is periodic with respect to $y$ and
    $
    %\norm{w}_{L^2_{\per}\left(Y, C(\bar{\Omega})\right)} \coloneqq 
    %\left[
    \int_{Y} \left(\sup_{x \in \bar{\Omega}} \abs{w(x,y)}\right)^p \dd y 
    %\right]^\frac{1}{2} 
    < \infty.
    $
  \end{itemize}

\begin{definition}[$L^p-$admissible test function]
\label{sec:two-scale-conv-1}
Let $1 \le p < + \infty$. A function $\psi \in
L^p(\Omega \times Y)$,  $Y$-periodic in the second component, is called an $L^p-$admissible test function if for
all $\varepsilon > 0$,
$\psi \left( \cdot, \frac{\cdot}{\varepsilon} \right)$ is measurable and
\begin{align}
\label{eq2sc:23}
\lim_{\varepsilon \to 0} \int_{\Omega} \abs{\psi \left( x,
  \frac{x}{\varepsilon} \right)}^p \dd x = \frac{1}{\abs{Y}}
  \int_{\Omega} \int_Y \abs{\psi (x,y)}^p \dd y \dd x.
\end{align}
\end{definition}

It is known that functions belonging to the spaces $\mathcal{D} \left(\Omega,
  C_\per^\infty (Y)\right)$, $C \left( \bar{\Omega}, C_{\per}(Y)
\right)$, $L^p_{\per}\left( Y,  C(\bar{\Omega})\right)$ or $L^p\left(\Omega, C_{\per}(Y)\right)$ are admissible \cite{allaireHomogenizationTwoscaleConvergence1992}, but the precise
characterization of those admissible test functions is still an
open question.
\begin{definition}
  A sequence $\{ v^\varepsilon \}_{\varepsilon>0}$ in $L^2(\Omega)$ is
  said to \emph{two-scale converge} to $v = v(x,y)$, with
  $v \in L^2 (\Omega \times Y)$, and we write
  $v^\varepsilon \tscale v$, if and only if:
  $\{ v^\varepsilon \}_{\varepsilon>0}$ is bounded and
\begin{align}
\label{eq2sc:2sc}
    \lim_{\varepsilon \to 0} \int_\Omega v^\varepsilon(x) \psi \left( x, \frac{x}{\varepsilon}\right) \dd x 
    = \frac{1}{\abs{Y}} \int_\Omega \int_Y v(x,y) \psi(x,y) \dd y \dd x,
\end{align}
for any test function $\psi = \psi (x, y)$, with $\psi \in
\mathcal{D} \left(\Omega, C_\per^\infty (Y)\right)$.
\end{definition}
In \eqref{eq2sc:2sc}, we can choose $\psi$ to be any ($L^2-$)admissible test
function. Any bounded sequence $v^{\varepsilon}\in L^2(\Omega)$ has
a subsequence that two-scale converges to a limit $v^0 \in L^2(\Omega
\times Y)$.

\begin{lemma}[Corrector result]\cite[Theorem 1.8, Remark 1.10 and
Corollary 5.4]{allaireHomogenizationTwoscaleConvergence1992}
\label{sec:two-scale-corrector}
Let $u^{\varepsilon}$ be a sequence of functions in $L^2(\Omega)$ that
two-scale converges to a limit $u^0 (x,y) \in L^2(\Omega\times
Y)$. Assume that 
\begin{align}
\label{eq2sc:24}
\lim_{\varepsilon \to 0} \norm{u^{\varepsilon}}_{L^2(\Omega)} =
  \norm{ 
  u^0
  }_{L^2(\Omega\times Y)}.
\end{align}
Then for any sequence $v^{\varepsilon}$ in $L^2(\Omega)$ that
two-scale converges to $v^0 \in L^2(\Omega \times Y),$ one has 
\begin{align}
\label{eq2sc:25}
u^{\varepsilon} v^{\varepsilon} \wcv \frac{1}{\abs{Y}} \int_Y u^0(x,y)
  v^0(x,y) \dd x \dd y \text{ in } \calD '(\Omega).
\end{align}
Furthermore, if $u^0(x,y)$ % is admissible,
belongs to $L^2 \left( \Omega, C_{\per}(Y) \right)$ or $L^2_{\per} \left( Y, C(\Omega) \right)$,
then 
\begin{align}
\label{eq2sc:29}
\lim_{\varepsilon \to 0}  \norm{u^{\varepsilon}(x) - u^0 \left( x,
  \frac{x}{\varepsilon} \right)}_{L^2(\Omega)} = 0.
\end{align}
\end{lemma}
In fact, the smoothness assumption on $u^0$ in \eqref{eq2sc:29}  is needed
only for $u^0 \left( x, \frac{x}{\varepsilon}
\right)$ to be measurable and to belong to $L^2(\Omega)$.

\section{Main results}
\label{sec:main-results}

% For each $\varepsilon > 0$ fixed, define the approximation problems as
% \begin{align}
%     \label{eq:4b}
%     \begin{cases}
%   -\Div \left( \bfa^{\varepsilon}\nabla u_n^{\varepsilon} \right) +
%     g ({u_n^{\varepsilon}}) 
%     &= T_n(f) \qquad \text{ in }\Omega,\\
%     u_n^{\varepsilon}
%     &= 0 \qquad \text{ on }\partial \Omega,
%     \end{cases}
% \end{align}
% where $T_n$ is the truncation operator defined as
% \begin{align}
% \label{eq:6b}
% T_n(x) = 
% \begin{cases}
% -n, & x \leq -n \\
% x, & |x| \leq n \\
% n, & x \geq n
% \end{cases}
% \quad \text{for } n > 0.
  % \end{align}
\begin{theorem}
\label{sec:main-results-1}
There exists a unique solution $u^{\varepsilon} \in H_0^1(\Omega)$ of
\eqref{eq:4a} such that
\begin{align}
\label{eq:32}
\int_{\Omega} \bfa^{\varepsilon} \nabla u^{\varepsilon} \nabla v
  + \int_{\Omega}
  g(u^{\varepsilon}) v
  = \int_{\Omega} f v, \qquad \text{ for all } v \in H_0^1(\Omega).
\end{align}

Moreover, the following estimate holds for all $\varepsilon > 0$:
\begin{align}
\label{eq:31}
  \norm{\nabla u^{\varepsilon}}_{L^2}^2
  \le C \left( \norm{f}_{L^{2_{*}}}^{2_{*}} + \abs{g(0)}^{2_{*}} + 1\right),
\end{align}
where $C > 0$ is independent of $\varepsilon$.
\end{theorem}
\begin{proof}
  \begin{enumerate}[wide]
  \item 
  Since $g \colon \RR \to \RR$ is continuous, there exists a sequence of locally Lipschitz functions $\gamma_m\colon \RR \to \RR$ such that $\gamma_m
\cv g$ uniformly as $m\to\infty$, cf. \cite[Theorem
6.2.1]{cobzasLipschitzFunctions2019}.
For each $\gamma_m$, for $x_*\le y_* \in \RR$ such that
$\gamma_m(x_{*}) \ge \gamma_m (y_{*})$, define $z_{*} = \sup \left\{ z
\ge x_{*} \colon \gamma_m(z) \le \gamma_m(x_{*})\right\}$ and define
\begin{align*}
  g_m := \sup \left\{ \gamma_{*}\colon \RR \to \RR \,\middle|\,
  \gamma_{*}(z) = \gamma_m(z) \text{ on } \RR \setminus [x_{*}, z_{*}],
  \text{ and } \gamma_{*}(z) = \gamma_m(x_{*}) \text{ otherwise} \right\}
\end{align*}

Since $\gamma_m$ is locally Lipschitz, so is $g_m$.  Since $g$ is increasing and $\gamma_m \cv g$ uniformly as $m\to\infty$, it follows that $g_m$ is also increasing and $g_m \cv g$ uniformly as $m\to\infty$. For each $\varepsilon > 0$ fixed and $m \in \NN$, define the approximation problems
\begin{align}
    \label{eq:4c}
    \begin{cases}
  -\Div \left( \bfa^{\varepsilon}\nabla u_m^{\varepsilon} \right) +
    g_m ({u_m^{\varepsilon}}) 
    &= f \qquad \text{ in }\Omega,\\
    u_m^{\varepsilon}
    &= 0 \qquad \text{ on }\partial \Omega.
    \end{cases}
\end{align}
By \cite[Theorem 4.7]{boccardoEllipticPartialDifferential2014},
problem \eqref{eq:4c} has a unique solution $u_m^{\varepsilon} \in
H_0^1(\Omega)$ such that $g_m(u_m^{\varepsilon}) \in L^1(\Omega)$
and 
\begin{align}
  \label{eq:25}
\int_{\Omega} \bfa^{\varepsilon} \nabla u_m^{\varepsilon} \nabla v
  + \int_{\Omega} g_m(u_m^{\varepsilon}) v
  = \int_{\Omega} f v, \quad \text{ for all } v \in H_0^1(\Omega)\cap L^{\infty}(\Omega).
\end{align}
We have the following inequality
\begin{align*}
\abs{ \int_{\Omega} g_m(u_m^{\varepsilon}) v -  \int_{\Omega}
  g(u_m^{\varepsilon}) v} \le \sup_{\RR} \abs{g_m -g} \int_{\Omega} \abs{v},
\end{align*}
so for $\delta > 0$ to be chosen later, there exists $M_{\delta}$ such that 
\begin{align*}
\abs{ \int_{\Omega} g_m(u_m^{\varepsilon}) v -  \int_{\Omega}
  g(u_m^{\varepsilon}) v} \le \delta \int_{\Omega} \abs{v},
\end{align*}
for all $m \ge M_{\delta}$. In particular,
\begin{align*}
\int_{\Omega}
  g(u_m^{\varepsilon}) v - \delta \int_{\Omega} \abs{v}
  \le \int_{\Omega} g_m(u_m^{\varepsilon})
  v
  \le
  \int_{\Omega}
  g(u_m^{\varepsilon}) v + \delta \int_{\Omega} \abs{v}.
\end{align*}
Therefore, for all $m \ge M_{\delta}$ and $v \in H_0^1(\Omega)\cap
L^{\infty}(\Omega)$, 
\begin{align}
\label{eq:24}
\int_{\Omega} \bfa^{\varepsilon} \nabla u_m^{\varepsilon} \nabla v
  + \int_{\Omega}
  g(u_m^{\varepsilon}) v - \delta \int_{\Omega} \abs{v}
  \le \int_{\Omega} f v,
\end{align}
and
\begin{align}
\label{eq:24b}
\int_{\Omega} \bfa^{\varepsilon} \nabla u_m^{\varepsilon} \nabla v
  + \int_{\Omega}
  g(u_m^{\varepsilon}) v + \delta \int_{\Omega} \abs{v}
  \ge \int_{\Omega} f v.
\end{align}

In \eqref{eq:24}, choose $v = v_k \in C_c^{\infty}(\Omega) \subset
H_0^1(\Omega) \cap L^{\infty}(\Omega)$ such that $v_k
\cv[H_0^1(\Omega)] u_m^{\varepsilon}$ and $v_k \cv[\text{a.e.}]
u_m^{\varepsilon}$ (up to a subsequence) as $k \to \infty$. Moreover,
by the Sobolev embedding theorem, we have $v_k \cv[L^{2^{*}}]
u_m^{\varepsilon}$ as $k \to \infty$, so we can assume  that
$\abs{v_k} \le V$ a.e. for some $V \in L^{2^{*}}(\Omega),$ up to a
subsequence \cite[Theorem 4.9]{brezisFunctionalAnalysisSobolev2011}. By the growth
condition \eqref{eq:16}, we have $g(u_m^{\varepsilon}) v_k \ge - C_{g} \left(
  \abs{u_m^{\varepsilon}}^{q-1} + h (x) \right) V$ which belongs to $L^1(\Omega)$ by the Sobolev embedding theorem. The following holds
\begin{align*}
 % \label{eq:24a}
  \begin{split}
    &\lambda \norm{\nabla u_m^{\varepsilon}}_{L^2}^2 - \int_{\Omega}
    \abs{g(0)} \abs{u_m^{\varepsilon}} - \delta \int_{\Omega} \abs{u_m^{\varepsilon}}\\
    &\coex[monotone \eqref{eq:5}]{\le} \lambda \norm{\nabla u_m^{\varepsilon}}_{L^2}^2 + \int_{\Omega}
  \left( g(u_m^{\varepsilon}) - g(0) \right)(u_m^{\varepsilon}-0) +
  \int_{\Omega} g(0) u_m^{\varepsilon} - \delta \int_{\Omega} \abs{u_m^{\varepsilon}}\\
    &\coex[ellipticity \ref{cond:a-elliptic}]{\le} \int_{\Omega} \bfa^{\varepsilon} \nabla u_m^{\varepsilon} \nabla u_m^{\varepsilon}
  + \int_{\Omega}
  g(u_m^{\varepsilon}) u_m^{\varepsilon} - \delta \int_{\Omega}
      \abs{u_m^{\varepsilon}}\\
    &\coex{=}\lim_{k \to \infty}\int_{\Omega} \bfa^{\varepsilon} \nabla u_m^{\varepsilon} \nabla v_k
  + \int_{\Omega}
  \liminf_{k \to \infty}  g(u_m^{\varepsilon}) v_k - \delta \lim_{k
      \to \infty} \int_{\Omega}  \abs{v_k}\\
  &\coex[Fatou]{\le}
\liminf_{k \to \infty} \left( \int_{\Omega} \bfa^{\varepsilon} \nabla u_m^{\varepsilon} \nabla v_k
  + \int_{\Omega}
  g(u_m^{\varepsilon}) v_k   -\delta
    \int_{\Omega} \abs{v_k}  \right)\\
  &\coex[\eqref{eq:24}]{\le}\liminf_{k \to \infty} \int_{\Omega} f v_k = \int_{\Omega}
    f u_m^{\varepsilon}, \quad \text{ for all } m \ge M_{\delta}.
  \end{split}
\end{align*}
Let $\tau > 0$ to be chosen later. The above estimate implies that, whenever $m \ge M_{\delta}$, we have 
\begin{align}
\label{eq:28}
  \begin{split}
    \lambda \norm{\nabla u_m^{\varepsilon}}_{L^2}^2
    &\coex{\le}  \int_{\Omega} \left( \abs{f} + 
      \abs{g(0)} + \delta \right) \abs{u_m^{\varepsilon}}\\
    &\coex[Young]{\le} \int_{\Omega} \left(
      \frac{1}{2_{*}}\left( \frac{\abs{f} + \abs{g(0)} + \delta}{\tau}
      \right)^{2_{*}} + \frac{1}{2^{*}}\left( \tau \abs{u_m^{\varepsilon}} \right)^{2^{*}}
      \right)\\
    &\coex[Sobolev]{\le} \frac{1}{2_{*}\tau^{2_{*}}} \int_{\Omega} 
      \left( {\abs{f} + \abs{g(0)} + \delta}
      \right)^{2_{*}}
      + \frac{1}{2^{*}} \tau^{2^{*}} C_S \norm{u_m^{\varepsilon}}_{H^1}^2\\
    &\coex[Poincar\'e]{\le} \frac{1}{2_{*}\tau^{2_{*}}} \int_{\Omega} 
      \left( {\abs{f} + \abs{g(0)} + \delta}
      \right)^{2_{*}}
      + \frac{1}{2^{*}} \tau^{2^{*}} C_S C_P\norm{\nabla u_m^{\varepsilon}}_{L^2}^2
  \end{split}
\end{align}
In \eqref{eq:28}, by first choosing $\delta = 1$, and then $\tau > 0$ such that $\frac{1}{2^{*}} \tau^{2^{*}} C_S C_P = \frac{1}{2}\lambda$,
we conclude that there exists $M_1 > 0$ and $C = C(\lambda, d, \Omega)
> 0$ such that whenever $m \ge M_1$, 
\begin{align}
\label{eq:29}
  \norm{\nabla u_m^{\varepsilon}}_{L^2}^2
  \le C \left( \norm{f}_{L^{2_{*}}}^{2_{*}} + \abs{g(0)}^{2_{*}} + 1\right).
\end{align}

\item From \eqref{eq:24} and \eqref{eq:24b}, for
  $m, m' \ge M_{\delta}$ we have
\begin{align}
\label{eq:26}
\int_{\Omega} \bfa^{\varepsilon} \nabla \left( u_m^{\varepsilon} - u_{m'}^{\varepsilon} \right) \nabla v
  + \int_{\Omega}
  \left(g(u_m^{\varepsilon})- g(u_{m'}^{\varepsilon}) \right) v
  \le 2\delta \int_{\Omega} \abs{v},
\end{align}
for any $v \in H_0^1(\Omega)\cap
L^{\infty}(\Omega)$. Let $\eta > 0$ to be chosen later, by
Poincar\'{e}  inequality,
$2\delta \int_{\Omega} \abs{ v} \le \int_{\Omega} \left(
  \frac{\delta}{\eta} \right)^2 + \left( \eta \abs{v} \right)^2 =
\frac{\delta^2}{\eta^2} \abs{\Omega} + \eta^2
\norm{v}^2_{L^2} \le \frac{\delta^2}{\eta^2} \abs{\Omega} + \eta^2
C_P^2\norm{\nabla v}^2_{L^2}$. Therefore, \eqref{eq:26} implies 
\begin{align*}
\int_{\Omega} \bfa^{\varepsilon} \nabla \left( u_m^{\varepsilon} - u_{m'}^{\varepsilon} \right) \nabla v
  + \int_{\Omega}
  \left(g(u_m^{\varepsilon})- g(u_{m'}^{\varepsilon}) \right) v -  \eta^2
C_P^2\norm{\nabla v}^2_{L^2}
  \le \frac{\delta^2}{\eta^2} \abs{\Omega},
\end{align*}
for any $v \in H_0^1(\Omega)\cap L^{\infty}(\Omega)$. In the above
estimate, in the same fashion as in the previous step, we choose $v= v_k \in C_c^{\infty}(\Omega)$ such that
$v_k \cv[H_0^1] u_m^{\varepsilon} - u_{m'}^{\varepsilon}$ as
$k \to \infty$ and $\abs{v_k} \le V$ for some
$V \in L^{2^{*}}(\Omega)$, to obtain
\begin{align*}
  \begin{split}
    &\lambda \norm{\nabla \left( u_m^{\varepsilon} - u_{m'}^{\varepsilon}
      \right)}_{L^2}^2 - \eta^2 C_P^2 \norm{\nabla \left(
      u_m^{\varepsilon} - u_{m'}^{\varepsilon} \right)}_{L^2}^2\\
    &\coex[ellipticity \ref{cond:a-elliptic}]{\le} \int_{\Omega} \bfa^{\varepsilon} \nabla \left( u_m^{\varepsilon} - u_{m'}^{\varepsilon} \right) \nabla \left( u_m^{\varepsilon} - u_{m'}^{\varepsilon} \right)
   -  \eta^2
C_P^2\norm{ \nabla \left( u_m^{\varepsilon} - u_{m'}^{\varepsilon} \right)}^2_{L^2}
    \\
    &\coex[monotone \eqref{eq:5}]{\le} \int_{\Omega} \bfa^{\varepsilon} \nabla \left( u_m^{\varepsilon} - u_{m'}^{\varepsilon} \right) \nabla \left( u_m^{\varepsilon} - u_{m'}^{\varepsilon} \right)
  + \int_{\Omega}
  \left(g(u_m^{\varepsilon})- g(u_{m'}^{\varepsilon}) \right) \left(
      u_m^{\varepsilon} - u_{m'}^{\varepsilon} \right)\\
    &\coex{\hphantom{\le}} \qquad -  \eta^2
      C_P^2\norm{\nabla \left( u_m^{\varepsilon} - u_{m'}^{\varepsilon} \right)}^2_{L^2}\\
   &\coex[LDCT and \eqref{eq:16}]{\le}\lim_{k \to \infty} \left( \int_{\Omega} \bfa^{\varepsilon} \nabla \left( u_m^{\varepsilon} - u_{m'}^{\varepsilon} \right) \nabla v_k
  + \int_{\Omega}
  \left(g(u_m^{\varepsilon})- g(u_{m'}^{\varepsilon}) \right) v_k -  \eta^2
C_P^2\norm{\nabla v_k}^2_{L^2} \right)\\
  &\coex{\le} \frac{\delta^2}{\eta^2} \abs{\Omega}
  \end{split}
\end{align*}

Choosing $\eta > 0$ such that $\eta^2 C_p^2 = \frac{\lambda}{2}$, we get 
\begin{align*}
  \frac{\lambda}{2} \norm{\nabla \left( u_m^{\varepsilon} - u_{m'}^{\varepsilon}
  \right)}_{L^2}^2
  \le \delta^2 \frac{2 C_P^2\abs{\Omega}}{\lambda}, \quad \text{ for
  all } m, m' \ge M_{\delta}.
\end{align*}
Since $\delta > 0$ is arbitrary, we conclude that
$(u_m^{\varepsilon})_{m \in \NN}$ is a Cauchy sequence in
$H_0^1(\Omega)$, as such, $u_m^{\varepsilon} \cv[H_0^1]
u^{\varepsilon}$ as $m \to \infty$. From \eqref{eq:29}, we conclude 
\begin{align}
\label{eq:31b}
  \norm{\nabla u^{\varepsilon}}_{L^2}^2
  \le C \left( \norm{f}_{L^{2_{*}}}^{2_{*}} + \abs{g(0)}^{2_{*}} + 1\right).
\end{align}

\item Fix $v \in H_0^1(\Omega)\cap
L^{\infty}(\Omega)$ in \eqref{eq:24} and \eqref{eq:24b}, and let $m
\to \infty$.  Using the Lebesgue dominated convergence theorem, we obtain that
\begin{align}
\label{eq:24c}
\int_{\Omega} \bfa^{\varepsilon} \nabla u^{\varepsilon} \nabla v
  + \int_{\Omega}
  g(u^{\varepsilon}) v - \delta \int_{\Omega} \abs{v}
  \le \int_{\Omega} f v,
\end{align}
and
\begin{align}
\label{eq:24d}
\int_{\Omega} \bfa^{\varepsilon} \nabla u^{\varepsilon} \nabla v
  + \int_{\Omega}
  g(u^{\varepsilon}) v + \delta \int_{\Omega} \abs{v}
  \ge \int_{\Omega} f v.
\end{align}
Since $\delta > 0$ is arbitrary, we conclude
\begin{align*}
\int_{\Omega} \bfa^{\varepsilon} \nabla u^{\varepsilon} \nabla v
  + \int_{\Omega}
  g(u^{\varepsilon}) v
  = \int_{\Omega} f v, \qquad \text{ for all } v \in H_0^1(\Omega)\cap
L^{\infty}(\Omega).
\end{align*}
Using density and the Lebesgue dominated convergence theorem, we obtain
\begin{align}
\label{eq:32a}
\int_{\Omega} \bfa^{\varepsilon} \nabla u^{\varepsilon} \nabla v
  + \int_{\Omega}
  g(u^{\varepsilon}) v
  = \int_{\Omega} f v, \qquad \text{ for all } v \in H_0^1(\Omega).
\end{align}
Note that the argument in Step 2 also shows that $u^{\varepsilon}$
is unique.
\end{enumerate}
\end{proof}

\begin{theorem}
\label{sec:main-results-2}
There exists $u^0\in H_0^1(\Omega)$ such that $u^{\varepsilon}
\wcv[H_0^1] u^0$ as $\varepsilon \to 0$, where $u^0$ is the solution
of the homogenized equation
\begin{align}
\label{eq:23a}
-\Div \left( \bfa^0 \nabla u^0 \right) + g (u^0) = f,
\end{align}
where $\bfa^0$ is given by
\begin{align}
\label{eq:10}
\bfa^0_{ij} \coloneqq \int_Y \bfa \left( \e^i +\nabla \chi^i \right) (\e^j +
  \nabla \chi^j) \dd y,
\end{align}
where $\chi^i \in H_{\per}^1(Y)/\RR$ are the solutions of the cell problems
\begin{align}
\label{eq:22a}
-\Div \bfa (y) \left( \nabla \chi^i + \e^i \right) = 0.
\end{align}
\end{theorem}
\begin{proof} Since we have \eqref{eq:31}, by  \cite[Proposition
  1.14.] {allaireHomogenizationTwoscaleConvergence1992}, there exist $u^0 \in H_0^1(\Omega)$, $u^1
\in L^2 \left( \Omega, H_{\per}^1(Y)/\RR\right)$, and $V \in L^{2^{*}}(\Omega)$ such that, as $\varepsilon\to0$, we have
\begin{align}
\label{eq:18}
  \begin{split}
    &u^{\varepsilon} \wcv[H_0^1] u^0,\\
  &\nabla u^{\varepsilon} \wcv[2] \nabla u^0 + \nabla_y u^1,\\
  &u^{\varepsilon} \cv u^0 \text{ a.e., and } \abs{u^{\varepsilon}} \le
    V \text{ a.e.}
  \end{split}
\end{align}

Recall the weak formulation of \eqref{eq:4a}: 
\begin{align}
\label{eq:19}
  \int_{\Omega}\bfa^{\varepsilon} \nabla u^{\varepsilon} \nabla v
  + \int_{\Omega} g (u^{\varepsilon}) v
  = \int_{\Omega} f v, \qquad \text{ for all } v \in C_c^{\infty}(\Omega).
\end{align}

Above, choose $v(x) = v^0(x) + \varepsilon v^1 \left( x,
  \frac{x}{\varepsilon} \right),$ for $v^0 \in C_c^{\infty}(\Omega)$
and $v^1 \in C_c^{\infty} \left( \Omega, C_{\per}^{\infty}(Y)/\RR \right)$, to get
\begin{align}
\label{eq:19b}
  \int_{\Omega}\bfa^{\varepsilon} \nabla u^{\varepsilon} \left( \nabla
  v^0 + \varepsilon \nabla v^1 + \nabla_y v^1\right)
  + \int_{\Omega} g (u^{\varepsilon}) (v^0 + \varepsilon v^1)
  = \int_{\Omega} f (v^0 + \varepsilon v^1).
\end{align}
Letting $\varepsilon \to 0 $, using \eqref{eq:18} and the Lebesgue
dominated convergence theorem, we obtain 
\begin{align}
\label{eq:20}
\int_{\Omega \times Y} \bfa(y) \left( \nabla u^0 + \nabla_y u^1
  \right) \left( \nabla v^0 + \nabla_y v^1 \right) \dd x \dd y
  + \int_{\Omega} g (u^0) v^{0}
  = \int_{\Omega} f v^0.
\end{align}

To read off the cell problems, we eliminate the macroscopic behavior
by choosing $v^0 = 0$, then \eqref{eq:20}
becomes 
\begin{align}
\label{eq:21}
\int_{\Omega \times Y} \bfa(y) \left( \nabla u^0 + \nabla_y u^1
  \right)  \nabla_y v^1  \dd x \dd y
  = 0.
\end{align}
Consider the ansatz $u^1 (x,y) =  \frac{\partial u^0}{\partial x_i}(x)
\chi^i (y)$.
Taking $v^1 (x,y) = \phi (x) \eta (y)$ for $\phi \in
C_c^{\infty}(\Omega)$ and $\eta \in C_{\per}^{\infty}(Y)/\RR$, we
conclude 
\begin{align}
\label{eq:22}
-\Div \bfa (y) \left( \nabla \chi^i + \e^i \right) = 0.
\end{align}
By substituting $v^1 = 0$ and  $u^1 (x,y) =  \frac{\partial u^0}{\partial x_i}(x)
\chi^i (y)$ into \eqref{eq:20}, and then using \eqref{eq:10}, we obtain the homogenized equation
\begin{align}
\label{eq:23}
-\Div \left( \bfa^0 \nabla u^0 \right) + g (u^0) = f.
\end{align}
\end{proof}

\begin{theorem}
\label{sec:main-results-3}
Suppose further that $f \in L^s(\Omega)$, with $s > \frac{d}{2}$. The following corrector result holds: 
\begin{align}
\label{eq:9}
\lim_{\varepsilon \to 0} \norm{u^{\varepsilon} -  u^0 -
  \varepsilon u^1 \left( \cdot, \frac{\cdot}{\varepsilon} \right)}_{H_0^1}
  =0,
\end{align}
where $u^1 (x,y) = \frac{\partial u^0}{ \partial x_i} (x) \chi^i (y). $
\end{theorem}
\begin{proof}
  We have
  \begin{align*}
    &\lambda \norm{\nabla u^{\varepsilon} -  \nabla u^0 - \nabla_y u^1
      \left( \cdot, \frac{\cdot}{\varepsilon} \right)
      }^2_{L^2}\\
    &\coex[\ref{cond:a-elliptic} and \eqref{eq:5}]{\le}
      \int_{\Omega} \bfa^{\varepsilon} \left( \nabla u^{\varepsilon} -
      \nabla u^0 - \nabla_y u^1 \right) \left( \nabla u^{\varepsilon} -
      \nabla u^0 - \nabla_y u^1 \right)\\
    &\coex{} \qquad + \int_{\Omega} \left( g(u^{\varepsilon}) - g(u^0
      + \varepsilon u^1) \right) (u^{\varepsilon} - u^0 - \varepsilon
      u^1)\\
    &\coex[\eqref{eq:32a}]{=} \int_{\Omega} f (u^{\varepsilon} - u^0
      - \varepsilon u^1)\\
    &\coex{}\qquad - \int_{\Omega} \bfa^{\varepsilon} \left(  
      \nabla u^0 + \nabla_y u^1 \right) \left( \nabla u^{\varepsilon} -
      \nabla u^0 - \nabla_y u^1 \right)\\
    &\coex{}\qquad - \int_{\Omega} g(u^0
      + \varepsilon u^1) (u^{\varepsilon} - u^0 - \varepsilon
      u^1)\\
    &\coex{\eqqcolon} I_1(\varepsilon) - I_2 (\varepsilon)- I_3(\varepsilon) .
  \end{align*}
%   Recall $u^1 (x,y) =  \frac{\partial u^0}{\partial x_i}(x)
  %   \chi^i (y)$ from the proof of \cref{sec:main-results-2}.
  Since $\chi^i$ is the solution of \eqref{eq:22a}, Lax-Milgram
  theorem implies that there exists $C = C(\lambda,\Lambda) > 0$ such
  that 
  \begin{align}
  \label{eq:11}
    \norm{\nabla \chi^i}_{L^2(Y)}
    \le C.
  \end{align}
Applying \cite[Corollary 12]{brezisSemilinearSecondorderElliptic1973}
for \eqref{eq:23a} implies that $u^0 \in W^{2,s} \subset
C^0$. Therefore, $u^1(x,y)$ is admissible and $u^1 \left(
  \cdot, \frac{\cdot}{\varepsilon} \right)$ is bounded in
$L^2(\Omega)$. Together with $u^{\varepsilon} \wcv[H_0^1] u^0$, we
conclude that 
\begin{align}
\label{eq:12}
\lim_{\varepsilon \to 0} I_1(\varepsilon) = \lim_{\varepsilon \to 0}
  I_3(\varepsilon) = 0.
\end{align}
  Now, as for $I_2(\varepsilon)$, we have 
  \begin{align*}
    I_2(\varepsilon)
    &= \int_{\Omega} \bfa^{\varepsilon} \left(  
      \nabla u^0 + \nabla_y u^1 \right) \left( \nabla u^{\varepsilon} -
    \nabla u^0 - \nabla_y u^1 \right)\\
    &=\int_{\Omega} \bfa^{\varepsilon} \left(  
      \nabla u^0 + \nabla_y u^1 \right)  \nabla u^{\varepsilon}
      - \int_{\Omega} \bfa^{\varepsilon} \left(  
      \nabla u^0 + \nabla_y u^1 \right) \left( 
      \nabla u^0 + \nabla_y u^1 \right)\\
    &\cv[\varepsilon \to 0] \int_{\Omega \times Y} \bfa(y) \left(  
      \nabla u^0 + \nabla_y u^1 \right) \left( 
      \nabla u^0 + \nabla_y u^1 \right)\\
    & \qquad \quad
      - \int_{\Omega \times Y} \bfa(y) \left(  
      \nabla u^0 + \nabla_y u^1 \right) \left( 
      \nabla u^0 + \nabla_y u^1 \right)\\
   &= 0,
  \end{align*}
  where we have used $\nabla u^{\varepsilon} \wcv[2] \nabla u^0 +
  \nabla_y u^1$ for the first integral and the averaging lemma \cite{bensoussanAsymptoticAnalysisPeriodic2011} for the last integral.  The corrector convergence \eqref{eq:9} now follows.
\end{proof}

\section{Conclusions}\label{sec:Conclusions}
This paper contributes to the theory of semilinear elliptic equations
by considering the homogenization of a class of second-order
semilinear elliptic equations characterized by discontinuous,
periodically oscillating coefficients, and monotone
nonlinearities. Motivated by physical models such as wave propagation
in composite media and reaction–diffusion phenomena, the study
addressed analytical challenges posed by the interplay of nonlinearity
and heterogeneous media.  Building on established homogenization
techniques, we proved the existence and uniqueness of solutions to the
fine-scale problem in Theorem \ref{sec:main-results-1}, established
the corresponding homogenized equation that governs the macroscopic
behavior of the system in Theorem \ref{sec:main-results-2}, and
constructed, using a higher regularity assumption, a first-order
corrector in Theorem \ref{sec:main-results-3}.  These findings help lay the groundwork for applying homogenization techniques to a broader class of nonlinear systems encountered in applied science and engineering.

\bibliographystyle{spmpsci}
\bibliography{homogenisation,arxiv}

\end{document}

%% file: figures/oneway-paper.pdf_tex
%% Creator: Inkscape 1.0 (4035a4fb49, 2020-05-01), www.inkscape.org
%% Accompanies image file 'oneway-paper.pdf' (pdf, eps, ps)
\begingroup%
  \makeatletter%
  \providecommand\color[2][]{%
    \errmessage{(Inkscape) Color is used for the text in Inkscape, but the package 'color.sty' is not loaded}%
    \renewcommand\color[2][]{}%
  }%
  \providecommand\transparent[1]{%
    \errmessage{(Inkscape) Transparency is used (non-zero) for the text in Inkscape, but the package 'transparent.sty' is not loaded}%
    \renewcommand\transparent[1]{}%
  }%
  \providecommand\rotatebox[2]{#2}%
  \newcommand*\fsize{\dimexpr\f@size pt\relax}%
  \newcommand*\lineheight[1]{\fontsize{\fsize}{#1\fsize}\selectfont}%
  \ifx\svgwidth\undefined%
    \setlength{\unitlength}{575.7832489bp}%
    \ifx\svgscale\undefined%
      \relax%
    \else%
      \setlength{\unitlength}{\unitlength * \real{\svgscale}}%
    \fi%
  \else%
    \setlength{\unitlength}{\svgwidth}%
  \fi%
  \global\let\svgwidth\undefined%
  \global\let\svgscale\undefined%
  \makeatother%
  \begin{picture}(1,0.63240555)%
    \lineheight{1}%
    \setlength\tabcolsep{0pt}%
    \put(0,0){\includegraphics[width=\unitlength,page=1]{oneway-paper.pdf}}%
    \put(0.18798321,0.49507706){\color[rgb]{0,0,0}\makebox(0,0)[lt]{\lineheight{1.25}\smash{\begin{tabular}[t]{l}$\nn$\end{tabular}}}}%
    \put(0.08878394,0.25841837){\color[rgb]{0,0,0}\makebox(0,0)[lt]{\lineheight{1.25}\smash{\begin{tabular}[t]{l}$Y_{s}$\end{tabular}}}}%
    \put(0,0){\includegraphics[width=\unitlength,page=2]{oneway-paper.pdf}}%
    \put(0.48575168,0.44322711){\color[rgb]{0,0,0}\makebox(0,0)[lt]{\lineheight{1.25}\smash{\begin{tabular}[t]{l}$\varepsilon$\end{tabular}}}}%
    \put(0,0){\includegraphics[width=\unitlength,page=3]{oneway-paper.pdf}}%
    \put(0.2743749,0.13847476){\color[rgb]{0,0,0}\makebox(0,0)[lt]{\lineheight{1.25}\smash{\begin{tabular}[t]{l}$\Omega^{\varepsilon}_{f}$\end{tabular}}}}%
    \put(0.26623382,0.01228795){\color[rgb]{0,0,0}\makebox(0,0)[lt]{\lineheight{1.25}\smash{\begin{tabular}[t]{l}$\Omega^{\varepsilon}_{s}$\end{tabular}}}}%
    \put(0.80965124,0.39084837){\color[rgb]{0,0,0}\makebox(0,0)[lt]{\lineheight{1.25}\smash{\begin{tabular}[t]{l}$\Omega$\end{tabular}}}}%
    \put(0,0){\includegraphics[width=\unitlength,page=4]{oneway-paper.pdf}}%
    \put(0.25474841,0.59281953){\color[rgb]{0,0,0}\makebox(0,0)[lt]{\lineheight{1.25}\smash{\begin{tabular}[t]{l}$Y$\end{tabular}}}}%
    \put(0,0){\includegraphics[width=\unitlength,page=5]{oneway-paper.pdf}}%
  \end{picture}%
\endgroup%